\UseRawInputEncoding 
\documentclass[aap]{imsart}

\RequirePackage{amsthm,amsmath,amsfonts,amssymb}
\RequirePackage[numbers]{natbib}
\RequirePackage[colorlinks,citecolor=blue,urlcolor=blue]{hyperref}
\RequirePackage{graphicx}

\startlocaldefs
\theoremstyle{plain}

\newtheorem{theorem}{Theorem}[section]
\newtheorem{lemma}[theorem]{Lemma}
\newtheorem{assum}[theorem]{Assumption}
\newtheorem{prop}[theorem]{Proposition}
\newtheorem{cor}[theorem]{Corollary}
\newtheorem{rem}[theorem]{Remark}
\theoremstyle{remark}
\newtheorem{definition}[theorem]{Definition}
\newtheorem*{example}{Example}

\newcommand{\E}{\mathbb{E}}
\newcommand{\ft}{\beta(t,\rho_{X_t})}

\newcommand{\gt}{a(t,\Gamma_{X_t})}

\endlocaldefs

\begin{document}
\begin{frontmatter}
\title{The Bismut-Elworthy-Li formula for semi-linear distribution-dependent SDEs driven by fractional Brownian motion}
\runtitle{The Bismut formula and  fractional mean-field SDEs}
\begin{aug}
\author[A]{\fnms{ Mahdieh }~\snm{Tahmasebi}\ead[label=e1]{tahmasebi@modares.ac.ir}},
\address[A]{Department of Applied Mathematics, Tarbiat Modares University, P.O. Box 14115-134, Tehran, Iran.\printead[presep={,\ }]{e1}}
\end{aug}

\begin{abstract}
In this work, we will show the existence, uniqueness, and weak differentiability of the solution of semi-linear mean-field stochastic differential equations driven by fractional Brownian motion. We prove an extension of the Bismut-Elworthy-Li formula and show some applications in the sensitivity analysis of variance swaps and the price of derivatives with respect to the initial point.
\end{abstract}



\end{frontmatter}
%

\section{Introduction}
Mean-field stochastic
differential equations (mean-field SDEs) are an extension of stochastic differential equations where the coefficients are allowed to depend on the law of the solution. Based on the works of Vlasov \cite{12}, Kac \cite{10} and McKean \cite{11}, mean-field SDEs, also referred to as McKean-Vlasov equations or as distribution dependent equations, arisen from
Boltzmann equation in physics, is used to model weak interaction between
particles in a multi-particle system. Henceforth the applications of mean-field problem and the study of dependent distribution SDEs attracted wide attention, (see for instance, \cite{huang5, carmona, huang, huang4}).
In \cite{5}, Lasry and Lions have considered the mean-field SDEs as an application in Economics and Finance,
The mean-field approach also has some application in systemic risk modeling,
especially for inter-bank lending and borrowing models, (see e.g., \cite{3, 4, 2, 1}).\\
For the case where the mean-field SDEs have additive noise, Jourdain proves in \cite{9} the existence of a unique weak solution under the assumptions of a bounded drift which is Lipschitz continuous in the law variable.
In \cite{13} and \cite{7} the existence of strong solutions of
mean-field SDEs are shown in the special form of drift which fulfills certain linear growth and Lipschitz conditions.
Recently, Bauer et. al. investigated the existence, uniqueness, and regularity properties of (strong)
solutions of one-dimensional mean-field SDEs when the
drift coefficient is allowed to be irregular in \cite{15} and \cite{14}. The regularity of the solution of mean-field SDEs also discussed in many articles. For instance, continuous differentiability of the solution $X_t^x$ with respect
to the initial point $x$, where the mean-field dependence is given
via an expectation functional of the form $\E(\phi(X_t^x))$,
has been shown in \cite{buckda2} under the assumption that the drift coeficient and $\phi: \mathbb{R} \rightarrow \mathbb{R}$ are continuously differentiable with bounded Lipschitz derivatives.
Also, in \cite{14}, the authors deal with the important case that the drift depends on the distribution function of the solution for the indicator function $\phi(z) = 1_{z\leq u}$. In this manuscript, we are interested in the mean-field SDEs whose distribution dependence of the solution is stated by an expectation functional form.\\
On the other hand, the Bismut-Elworthy-Li formula (BEL formula) gives a representation of expressions of the form
$$\frac{d}{dx} \E[\Phi(X_t^x)]$$
for $\Phi : \mathbb{R}^d \rightarrow \mathbb{R}$, $d \geq 1$ and for a desirable stochastic process $X_t^x$, which doesn't involve the derivative of $\Phi$.  The Bismut-Elworthy-Li formula, also known as the Bismut formula, based on Malliavin calculus,  is a very effective tool in the analysis of distributional regularity for various stochastic
models, with additive noise and multiplicative noise  (see e.g., \cite{wang, huang2, huang3}.  The Bismut formula for multi-dimensional mean-field SDEs with multiplicative noise and smooth drift
and volatility coefficients has been studied in \cite{banos}. Bauer et. al. in \cite{14} have derived a Bismut-Elworthy-Li type formula for mean-field SDEs under additive noise where the dependence of the drift $b$ on the solution law is of the form
\begin{equation*}
dX_t = b(t, X_t, \int_0^t \phi(t, X_t, z) dP_{X_t}(z) ) dt + dB_t, \quad X_0=x \in \mathbb{R},
\end{equation*}
for some $b, \phi : [0,T] \times \mathbb{R} \times \mathbb{R} \rightarrow \mathbb{R}$ and $B_t$ is a Wiener process. 
Also, the Bismut formula was given
for distribution-path dependent SDEs with distribution-free noise in \cite{bao}.\\
In this article, we are concerned with the equations driven by fractional Brownian motions. In fractional literature, in 2017, Fan \cite{fan} investigated a Bismut-Elworthy-Li type formula for the distribution-free SDEs with additive noise driven by fractional Brownian motion, assuming the Lipschitz continuous property on the drift.  
We also mention the interesting work \cite{amine}, in 2018, where Amine et. al. derived the Bismut-Elworthy-Li type formula for the distribution-free SDEs with given additive fractional noise in rough path sense with Hurst parameter $H < \frac12$ and singular drift coefficients for the modeling of regime switching effects in stock markets. \\
The mean field SDEs driven by fractional Brownian motion introduced by Buckdhan and Jing in \cite{buckda}. They have considered distribution-dependent SDEs of the form 
\begin{equation}\label{equa2c}
dX_{t}=b(t, X_t, \mathcal{L}_{X_t} )dt+ (\gamma_s X_s + \sigma(t,\mathcal{L}_{X_t})\})dW_{t}^{H} ,     \qquad t\in [0,T],
\end{equation}
where $ \mathcal{L}_{X_t}$ is the law of $X_t$, $H > \frac12$ and $W_t^H$ is $d$-dimensional fractional Brownian motion, under the Lipschitz property on diffusion coefficient $\sigma$ and the drift coefficient $b$, with Wasserstein measure $\mathbb{W}(.,.)$, as follows:
\begin{equation}\label{bcondi}
\vert b(t, x, \mu) - b(t, y, \nu)| \leq K(t)(\vert x - y\vert +\mathbb{W}(\mu, \nu)), \quad \vert \sigma(t, \mu) - \sigma(t, \nu)\vert \leq  K(t)\mathbb{W}(\mu, \nu),
\end{equation}
where $x,y \in \mathbb{R}^d$ and $\mu, \nu$ belongs to the space of probability measures on $\mathbb{R}^d$ with finite $\theta$-th moment, $\theta \geq 1$, also the stochastic integral is in the Skorohkod sense. They proved the existence and uniqueness
of the strong solution of that SDE in some special Banach space, denoted by $ L^{2,*}([0, T] \times \mathbb{R}^d)$ (see Section 3). \\
In 2021, Fan  et. al. \cite{fanhuang} proved the BEL formula for distribution dependent SDEs of the form (\ref{equa2c}) when the function $\gamma$ is zero. \\
We extend the result in \cite{fanhuang} and \cite{buckda} for the case, where the condition (\ref{bcondi}) is not covered our assumption on $b$ in this paper and also we allow the multiplicative noise driven by fractional Brownian motions with Hurst parameter $H > \frac12$.\\
In this article, we are interested in multiplicative stochastic differential equations driven by
fractional Brownian motion of the form 
\begin{equation}\label{equa1c}
dX_{t}=\{b(t,\rho_{X_t)})X_{t} +\beta(t,\rho_{X_t})\}dt+\{ C_t X_{t}+a(t,\Gamma_{X_t})\}dW_{t}^{H} ,     \qquad t\in [0,T] ,  \quad X_0=x > 0,
\end{equation}
where $\rho_{X_t}:=\mathbb{E}(\varphi(X_{t}))$ and $\Gamma(X_t)= \mathbb{E}(\psi(X_{t}))$, in which the functions $\varphi:\mathbb{R}^{d}\rightarrow \mathbb{R}^{d}$, $\psi:\mathbb{R}^{d}\rightarrow \mathbb{R}^{d}$,
the functions $b,\beta$ and $a$  from  $[0,T]\times \mathbb{R}^{d}$ to $\mathbb{R}^{d}$ are measurable, $C_.$ is a bounded deterministic function and $T\in\mathbb{R}, T>0$. We remark that the drift coefficient of (\ref{equa1c}) is not satisfies (\ref{bcondi}) and is an extension of this type of dependent distribution SDEs.  Also, we should study the fractional integrals in the concept of Skorokhod integrals. \\
The first contribution to this manuscript is to establish existence, uniqueness, and weak differentiability 
of strong solutions of this type of mean-field SDEs and the second main contribution is to establish the BEL formula in this multiplicative fractional distribution-dependent SDEs. To achieve this purpose, usually one needs to show that the flow of the solution is its Malliavin derivative in some direction. But for the solution of (\ref{equa1c}), this relationship can not be covered unless one additional term appears in the formula, (see  Theorem \ref{mt}). To handle this term in the BEL formula, we need that its Malliavin derivative in some direction would be nonzero, however, the solution can meet the zero during their path. To overcome this problem we proceed on some stopping time to show that the fundamental solution of (\ref{equa1c}) can be presented by the Mallivain derivative of the solution in some direction which is in the domain of the Skorokhod operator. (see Theorem \ref{vaznn}). \\
This article is organized as follows: Section 2 is devoted to recalling some useful facts
on fractional Brownian motion, fractional calculus and Girsanov theorem and Malliavin calculus for stochastic fractional integrals. In Section 3, we prove the existence and uniqueness of the solutions of Mean-field SDEs with multiplicative fractional noise when the drift coefficient is taken as a product of the function of the law of the solution and the solution itself. The regularity of the solution and obtaining a representation for the Malliavin drivative are shown in Section 4. In Section 5, the existence of the flow of the solution and a transfer principle between the flow and the Malliavin derivative of the solution is derived. The Bismut formula is established in section 6 and finally, we take some applications in finance in the last section. 
\section{Preliminaries}
In this section, we recall some concepts of fractional calculus, Girsanov theorem, and Malliavin calculus for stochastic fractional integrals. For more detailes, the reader can see the exhaustive references \cite{oksendal, Mishura}.
\subsection{Fractional operator}
Let $a,b \in \mathbb{R}$ and $a < b$. For every $f \in L^1([a,b];\mathbb{R})$ and $\alpha >0$, the left-sided and the right-sided fractional Riemann-Liouville integral of $f$ of order $\alpha$ are defined as   
\begin{equation*}
I_{a^+}^{\alpha} f(x) = \frac{1}{\Gamma(\alpha)}\int_a^x \frac{f(y)}{(x-y)^{1-\alpha}} dy, 
\end{equation*}
\begin{equation*}
I_{b^-}^{\alpha} f(x) = \frac{(-1)^{-\alpha}}{\Gamma(\alpha)}\int_x^b \frac{f(y)}{(x-y)^{1-\alpha}} dy, 
\end{equation*}
where $x \in (a,b)$ and $\Gamma$ denotes the Gamma function. For $\alpha \in (0,1)$ and $p \geq 1$, fractional differentiation is shown as an inverse operation. If $f \in I_{a^+}^{\alpha}( L^p([a,b];\mathbb{R}))$ (respectively, $f \in I_{b^-}^{\alpha}( L^p([a,b];\mathbb{R}))$), then there exists a unique function $g$ such that  $f = I_{a^+}^{\alpha}(g)$, (respectively, $f = I_{b^-}^{\alpha}(g)$). Samko et al. \cite{samko} provide a characterization of these functions coincides with the left-sided and the right-sided fractional derivative, defined as
\begin{equation*}
D_{a^+}^{\alpha} f(x) = \frac{1}{\Gamma(1-\alpha)}\Big(\frac{f(x)}{x^\alpha}+\alpha \int_a^x \frac{f(x)-f(y)}{(x-y)^{1+\alpha}} dy\Big), 
\end{equation*}
\begin{equation*}
D_{b^-}^{\alpha} f(x) = \frac{1}{\Gamma(1-\alpha)}\Big(\frac{f(x)}{(b-x)^\alpha}+\alpha \int_x^b \frac{f(x)-f(y)}{(x-y)^{1+\alpha}} dy\Big).
\end{equation*}
\subsection{Stochastic fractional integral}
Let $(\Omega, \mathcal{F}, P)$ be the Wiener space associated with fractional Brownian motion $\{W_t^H; t \in [0,T]\}$ with Hurst parameter $H \in (\frac12, 1)$ and consider $\{\mathcal{F}_t\}_{t \in [0,T]}$ as the complete filteration generated by $W^H$.\\
Let $\mathcal{H}$ be the Hilbert space defined as the closure of the set of step functions on $[0, T]$ with respect to the scalar product 
\begin{equation}\label{RH}
\langle 1_{[0,t]}, 1_{[0,s]} \rangle_{\mathcal{H}} = R_H(t,s) =\int_0^{t \wedge s} K_H(t,r)K_H(r,s)dr,
\end{equation}
where $K_H(t,s)$ is the square integrable kernel given by
\begin{equation*}
K_H(t,s)=  c_H s^{\frac12-H}\int_s^t (u-s)^{H-\frac32}u^{H-\frac12} du,
\end{equation*}
in which $c_H^2=\frac{H(2H-1)}{\beta(2-2H,H-\frac12)}$ and $t > s$.  Formula (\ref{RH}) provides the following explicit expression of $R_H(t,s)$
\begin{equation*} 
R_H(t,s) =\alpha_H \int_0^t \int_0^s \phi(r,u) dr du,
\end{equation*}
 where $\alpha_H = H(2H-1)$ and $\phi(t,s)= \vert t-s\vert^{2H-2}$. \\
The mapping $1_{[0,t]} \rightarrow W_t^H$ can be extended to an isometry between $\mathcal{H}$ and a Gaussian space associated with $W^H$, denoted by $\psi \rightarrow W^H(\psi)=\int_0^T \psi(t) dW_t^H$.\\
Besides, we define an operator $K_H$ on $\mathcal{H}$ as 
\begin{equation*}
(K_H\psi)(s)= c_H \Gamma(H-\frac12)s^{\frac12-H}I_{0^+}^{H-\frac12}(u^{H-\frac12}\psi(u))(s).
\end{equation*}
This operator from  $L^2([0,T]; \mathbb{R}^d)$ into $I_{0^+}^{H+\frac12} (L^2([0,T]; \mathbb{R}^d)$ is an isomorphism and it's inverse operator $K_H^{-1}$ is of the form 
\begin{equation*}
((K_H)^{-1}\psi)(s)= c_H \Gamma(H-\frac12)s^{H-\frac12}D_{0^+}^{H-\frac12}(u^{\frac12-H}\psi'(u))(s).
\end{equation*}
Adjoint operator $K_H^*$ on $\mathcal{H}$ is 
\begin{equation*}
(K_H^*\psi)(s)= c_H \Gamma(H-\frac12)s^{\frac12-H}I_{T^-}^{H-\frac12}(u^{H-\frac12}\psi(u))(s).
\end{equation*}
and its inverse operator is
\begin{equation*}
((K_H^*)^{-1}\psi)(s)=\frac{1}{ c_H \Gamma(H-\frac12)}s^{\frac12-H}D_{T^-}^{H-\frac12}(u^{H-\frac12}\psi(u))(s).
\end{equation*}
According to \cite{alos}, for every $u,v \in \mathcal{H}$, $\langle K_H^* u, K_H^* v \rangle_{L^2[0,T]} = \langle u, v\rangle_\mathcal{H}$ and the operator $K_H^*$ is an isometry between $\mathcal{H}$ and $L^2([0,T]; \mathbb{R}^d)$ and then the injection $R_H=K_H o K_H^*$  embeds $\mathcal{H}$ densely into $\Omega$ such that for any $\psi \in \Omega^* \subset\mathcal{H}$, $\E\Big(\exp\{iW^H(\psi)\}\Big)=\exp\{-\frac12 \Vert \psi \Vert_{\mathcal{H}}^2\}$.\\
Consider the process $W_t=W_t^H((K_H^*)^{-1}1_{[0,t]})$. It is a Wiener process and the filteration generated by $W$ coincides with $\{\mathcal{F}_t\}_{t \in [0, T]}$.
\subsection{Girsanov transformation}
Let $\gamma$ be a  bounded function in $\mathcal{H}$. For any $w\in \Omega$ and $t \in [0,T]$, we define the following operators 
\begin{equation*}
\mathcal{T}_t(w)= w+\int_0^{t\wedge .} K_H^* (\gamma 1_{[0, t]})(s) ds,
\end{equation*}  
\begin{equation*}
\mathcal{A}_t (w)=w- \int_0^{t\wedge .} K_H^* (\gamma 1_{[0, t]})(s) ds. 
\end{equation*}  
It is clear that $\mathcal{T}_t (\mathcal{A}_t)(w)=\mathcal{A}_t (\mathcal{T}_t)(w) = w$. Let 
\begin{equation*}
\mathcal{E}_t (\gamma):= \exp\{ \int_0^t \gamma(s)dW_s^H -\frac12 \int_0^t (K_H^*(\gamma1_{[0,t]}))^2(s) ds   \}.
\end{equation*} 
Hence 
\begin{equation*}
\mathcal{E}_t^{-1}(\mathcal{T}_t) (\gamma)= \exp\{- \int_0^t \gamma(s)dW_s^H -\frac12 \int_0^t (K_H^*(\gamma1_{[0,t]}))^2(s) ds   \}.
\end{equation*} 
Following a similar argument in Lemma 2.4 in \cite{jing}, we conclude that for all $p \geq 1$
\begin{equation}\label{bound}
\E\Big(\sup_{0 \leq t \leq T} \mathcal{E}_t^p(\mathcal{T}_t)\Big) < \infty \quad \E\Big(\sup_{0 \leq t \leq T} \mathcal{E}_t^p(\mathcal{T}_t)\Big) < \infty .
\end{equation}

\subsection{Malliavin calculus in fractional case}
 Denote by $ \mathcal{S} $  the set of smooth functionals of the form
 \begin{equation*}
F=f\left( W^H(u_{1}),W^H(u_{2}),\cdots W^H(u_{n})\right),
\end{equation*}
 where $ f \in C^{\infty}_{b} (\mathbb{R}^n)$ (f and all its derivatives are bounded) and $ u_{i} \in \mathcal{H},~ i=1,2,\cdots ,n$. For every $F \in \mathcal{S}$, define
\begin{equation*}
D^H F=\sum_{i=0}^{n} \frac{\partial f}{\partial x_{i}}\left(W^H(u_{1}), W^H(u_{2}),\cdots W^H(u_{n})\right)u_{i}.
\end{equation*}
 The derivative operator $D^H$ is a closable operator from $L^{p}(\Omega)$ into $L^{p}(\Omega,\mathcal{H})$ for every $p \geq 1$. We denote by $ \mathbb{D}^{1,p}$ the closure of $ \mathcal{S} $ respect to the norm
\begin{equation*}
\|F\|_{1,p}^p=E|F|^p+E\|D^H F\|^p_{\mathcal{H}}.
\end{equation*}
However, it is more convenient to use another Malliavin derivative for fractional Brownian motions defined in \cite{duncan}.\\
Duncan et al. \cite{duncan} has defined $\phi$-derivative of $F \in L^p$, $p \geq 1$, as a directional derivative of the random variable $F$ and 
denoted by $D_{s}^{\Phi}F$ in the following form 
\begin{equation*}
\langle D^{\Phi}F(w), g \rangle_{L^2}= \lim_{\epsilon \rightarrow 0} \frac{1}{\epsilon}\Big\{ F(w+\int_0^. \Phi g(u)du)-F(w)     \Big\},
\end{equation*}
where $\Phi g(u)=\int_0^T \phi(s,u)g(s)ds$ and $g \in \mathcal{H}$. For instance, Equation (3.6) of \cite{duncan} demonstrated that if $f:\mathbb{R}\longrightarrow \mathbb{R}$ be a twice continuously differentiable function with bounded second derivative, then 
$$D_{s}^{\phi}(\int_{0}^{\infty} f(u) dW_{u}^{H})=\int _{0}^{\infty}\phi(u,s) f(u)du. $$ 
 Recall $\mathcal{L}(0, T)$ as the family of Malliavin differentiable stochastic processes $G$ such that 
 $\mathbb{E}\| G\|_{\mathcal{H}}+\mathbb{E}\|D^{\phi} G\|_{\mathcal{H}\otimes \mathcal{H}}<\infty$, and for any sequence of partition $\pi:0=t^n_0 \leq \cdots \leq t_n^n=T$ of $[0,T]$ such that $\vert \pi \vert \rightarrow 0$ as $n \rightarrow \infty$,
\begin{equation*}
\sum_{i=0}^n \mathbb{E}\Big\{ \int_{t_{i}^n}^{t_{i+1}^n}\int_0^T( D_{r}^{\phi}G^\pi_{t_i^n} -D_{r}^{\phi} G_s)\phi(r,s) dr ds\Big\}^2+ \mathbb{E}\Big\{ \| G^\pi - G \|_H \Big\} < \infty.
\end{equation*}
The following proposition shows that how we can compute the $\phi$-derivative of stochastic integral. 
\begin{prop}(Duncan \cite{duncan}, Theorem \em{4.2}) If $(F_{t})_{t\in [0,T]}$ be a stochastic process in 
$\mathcal{L}(0,T)$ and $\sup\mathbb{E}(\vert\phi_{s}D_{s}^{\Phi}F_{s}\vert^{2})<\infty $,
then $\eta_{t} :=\int_{0}^{t} F(u) dW_u^H $ is $\phi$-differentiable and
\begin{equation}\label{dphi}
~ D^{\phi}_{s}\eta_{t}= \int_{0}^{t}D_{s}^{\phi} F_{u}dW_{u}^{H}+\int _{0}^{t}F_{u}\phi(s,u)du,  \qquad  a.s.
\end{equation}
\end{prop}
These facts  state  Ito formula for the sochastic process 
\begin{equation*}
\eta _{t}= \xi+\int _{0}^{t}G_{u}du +\int_{0}^{t}F_{u}dW_{u}^{H},\quad \xi\in\mathbb{R}, t\in[0,\tau].
\end{equation*}
\begin{prop}\label{ito}
{(Duncan \cite{duncan} Theorem \em{4.6})}
Let $(F_{s}^{i})_{s\in [0,T]}$, $i=0,...,n$, satisfying the condition of Theorem \em{4.3} of \cite{duncan} and $\xi _{t}^{i}=\int_{0}^{t}F_{s}^{i}dB_{s}^{H}$.
If
 $ f\in \mathbb{C}_{b}^{\frac{1}{2}}(\mathbb{R}_{+}\times \mathbb{R}^{n})$ 
and 
$ (f_{s_{k}}(s,\eta_{s})F_{s}^{k})_{s\in[0,T]}\in \mathcal{L}[0,T]$
then
\begin{align*}
f(t,\xi_{+}^{1},...,\xi_{+}^{n}) & =f(0,...,0)+\int _{0}^{t} f_{s} (u,\xi_{u}^{1},...,\xi_{u}^{n}) du\\
&  ~~ +\Sigma_{k=1}^{n}\int _{0}^{t}f_{x_{k}}(u,\xi_{u}^{1},...,\xi_{u}^{n})F_{u}^{k}dW_{u}^{H}\\
& ~~ +\Sigma_{k,l=1}^{n}\int f_{x_{k}x_{l}}(u,\xi_{u}^{1},...,\xi_{u}^{n})F_{u}^{k}D^\phi_{u}\xi _{u}^{\ell}du,\quad a.s.
\end{align*}
\end{prop}
Consequently, from Section 5.7 in \cite{Hu}, we have $$D_s^\phi(F)= K_HK_H^* D_s^HF.$$
Due to \cite{biagini}, for every stochastic process $F \in \mathbb{D}^{H}$, $\phi$-derivative of $F$ is expressed as 
\begin{equation}\label{relation1}
D_s^\phi(F)=\int _{0}^T D_r^H F \phi(r,s) dr .
\end{equation}
The adjoint operator of $D^H$, called $ \delta^{H}(.)$, can be characterized as a divergence operator associated to $D^\phi$ (See, for example Theorem 6.23  in \cite{Hu}, or Proposition 2.3 in \cite{Hu2}).
\begin{definition}
For any measurable function $G \in L^2([0,T]\times \Omega)$, we called $G$ is Skorohkod integrable if there exists $\delta^H(G) \in L^2(\Omega;\mathcal{H})$ such that for any $F \in D^{1,2}$,
\begin{equation*}
E\left[F\delta^{H}(G)\right]=\int_0^T E[D_t^{\phi}F G(t) ]dt =E\left\langle D^{H}F, G\right\rangle_{\mathcal{H}},
\end{equation*}
\end{definition}
Let us finish this subsection by giving a relation between Malliavin derivative and
divergence operators of both processes $W^H$ and $W$ that are needed later on.
\begin{prop}\label{relationDH}
(\cite{nualart} Proposition 5.2.1) For any $F \in D_W^{1,2} \cap D^{1,2}$,
$$K_H^* D^HF = DF,$$
where $D$ is the derivative operator with respect to the Wiener process $W$ and $D_W^{1,2}$ it's domain in $L^2(\Omega)$
\end{prop}
\begin{prop}\label{relationdeltaH}
(\cite{nualart} Proposition 5.2.2) $Dom(\delta^H)= (K_h^*)^{-1}(Dom \delta_W)$, and for every $\mathcal{H}$-valued random variable $u \in dom(\delta_W)$ we have $\delta^H(u)=\delta(K_H^*u)$, where $\delta_W$ is the dual operator of $D_W$.
\end{prop}
\section{Fractional Mean-Field SDEs}
In this section, we consider a mean-field type stochastic differential equation (SDE) driven by a fractional Brownian motion with Hurst parameter $H>\frac{1}{2}$ in the following form 
\begin{equation}\label{equa1}
dX_{t}=\{b(t,\rho_{X_t)})X_{t} +\beta(t,\rho_{X_t})\}dt+\{ C_t X_{t}+a(t,\Gamma_{X_t})\}dW_{t}^{H} ,     \qquad t\in [0,T] ,  \quad X_0=x > 0,
\end{equation}
where $\rho_{X_t}:=\mathbb{E}(\varphi(X_{t}))$ and $\Gamma(X_t)= \mathbb{E}(\psi(X_{t}))$, in which the functions $\varphi:\mathbb{R}^{d}\rightarrow \mathbb{R}^{d}$, and $\psi:\mathbb{R}^{d}\rightarrow \mathbb{R}^{d}$,
 the functions $b,\beta$ and $a$  from  $[0,T]\times \mathbb{R}^{d}$ to $\mathbb{R}^{d}$ are measurable, $C_.$ is a bounded deterministic function and $T\in\mathbb{R}, T>0$. 
For simplicity, we assume $d=1$ and we mention that a similar argument of all our analysis techniques are hold true for $d>1$.   \\
Denote $e_{t,X}=exp\Big\{-\int_0^t b(s,\rho_{{X}_s}) ds\Big\}$ and define 
 $$\tilde{\beta}(t,{\rho}_{{X}_t})=\beta(t,{\rho}_{{X}_t})e_{t,X}, \quad \quad \tilde{a}(t,{\Gamma}_{{X}_t})=a(t,{\Gamma}_{{X}_t})e_{t,X}.$$
\begin{assum}\label{assum}
 The functions $\varphi,\psi,b, a$ and $\beta$ satisfy the following conditions:
\begin{enumerate}
\item they are continuously differentiable with bounded Lipschitz constants $K$, 
\item the function $\tilde{a}(.,.)$ is a bounded function on $[0,T] \times \mathbb{R}$, 
\item the functions $\tilde{\beta}(.,.)$ has linear growth, i.e., for every $t \in [0,T]$ and a stochastic process $X$, we have
\begin{align*}
\vert \tilde{\beta}(t,\rho_{X_t})\vert \leq K\Big(1+ e_{t,X}\E{\vert X_t\vert}\Big),
\end{align*} 
\item the functions $\tilde{a}(.,.)$ and $\tilde{\beta}(.,.)$ are Lipschitz functions, i.e., for every $t,s \in [0,T]$ and stochastic processes $X$ and $Z$, we have 
\begin{align*}
\vert \tilde{\beta}(t,\rho_{X_t})-\tilde{\beta}(s,\rho_{Z_s}) \vert& +\vert \tilde{a}(t,\Gamma_{X_t})-\tilde{a}(s,\Gamma_{Z_s}) \vert \\
&\leq K\Big(\vert t-s\vert+\E\Big\vert e_{t,X}{X_t}-e_{s,Z}{Z_s} \Big\vert\Big).
\end{align*}
\end{enumerate}
\end{assum}
We first show that under Assumption \ref{assum}, the SDE (\ref{equa1}) has a unique solution in some Banach space.
Let $\mathcal{E}_t:= \mathcal{E}_t(C)$ and denote by $ \mathcal{L}^{2,*}([0,T],\mathbb{R})$ the Banach space of $\mathcal{F}_t$-adapted process $M_t$, $t \in [0,T]$, satisfying
 \begin{equation*}
\sup_{0 \leq t \leq T}\E\Big(e_{t, M_t}^2  M_t^2 \mathcal{E}^{-1}_t\Big) < \infty,
\end{equation*}
and ${L}^{2,*}([0,T],\mathbb{R})$ the Banach space of $\mathcal{F}_t$-adapted process $M_t$, $t \in [0,T]$, satisfying
 \begin{equation*}
\sup_{0 \leq t \leq T}\E\Big(  M_t^2 \mathcal{E}^{-1}_t\Big) < \infty.
\end{equation*}
\begin{definition}
A solution of Equation (\ref{equa1}) is a stochastic process $(X_t)_{t \in [0,T]} \in  \mathcal{L}^{2,*}([0,T],\mathbb{R})$ satisfying (\ref{equa1}) and $exp\Big\{-\int_0^. b(s,\rho_{{X}_s}) ds\Big\} X1_{[0,t]} \in Dom(\delta^H)$, $t \in [0,T]$.
\end{definition}
If, moreover, the function $b$ is uniformly bounded, the Cauchy-Schowartz inequality implies that $\mathcal{L}^{2,*}([0,T],\mathbb{R}) \subset {L}^{2,*}([0,T],\mathbb{R})$, which is exactly the case mentioned in \cite{buckda}. \\
Consider the following stochastic differential equation 
\begin{equation}\label{equa2}
d\tilde{X}_{t}=\tilde{\beta}(t,{\rho}_{\tilde{X}_t})dt+\{ C_t \tilde{X}_{t}+\tilde{a}(t,{\Gamma}_{\tilde{X}_t})\}dW_{t}^{H} ,     \qquad t\in [0,T] ,  \quad \tilde{X}_0=x,
\end{equation}
From Ito formula (Proposition \ref{ito}), the following lemma will be result, obviously.
\begin{lemma}\label{c1}
Under Assumption \ref{assum}, for every solution $X_t \in  \mathcal{L}^{2,*}$ of SDE (\ref{equa1}), the process  $e_{t,X} X_t \in L^{2,*}$  is a solution of SDE (\ref{equa2}). Conversely, for every solution $\tilde{ X}_t \in  {L}^{2,*}$ of SDE (\ref{equa2}), the process  $e_{t,\tilde{X}}\tilde{X}_t \in \mathcal{L}^{2,*}$  is a solution of SDE (\ref{equa1}).
\end{lemma}
Consequently, following a similar proof of Theorem 3.7 of \cite{buckda}, we verify that SDE (\ref{equa1}) has a unique solution in $X_t \in  \mathcal{L}^{2,*}$ and also for every two solutions $X_t$  and $Z_t$ of SDE (\ref{equa1}) with initial conditions $X_0=x$ and $Z_0=z$, 
\begin{equation*}
\E\Big(\Big\vert  e_{t,X} X_t -e_{t,Z}{Z}_t\Big\vert^2 \mathcal{E}_t^{-1}\Big) \leq c \vert x-z \vert^2 + c\int_0^t  \E\Big(\Big\vert  e_{s,X} X_s -e_{s,Z}{Z}_s\Big\vert^2 \mathcal{E}_s^{-1} \Big) ds,
\end{equation*}
where $c$ is a constant. Therefore, for some positive constant $c_0$
\begin{equation}\label{xz}
\E\Big(\Big\vert  e_{t,X} X_t -e_{t,Z}{Z}_t\Big\vert^2 \mathcal{E}_t^{-1}\Big) \leq c_0 \vert x-z \vert^2 .
\end{equation}
These facts stated in the following theorem.
\begin{theorem}\label{c2}
Under Assumption \ref{assum}, SDE (\ref{equa1}) has a unique solution $X_t \in  \mathcal{L}^{2,*}([0,T], \mathbb{R})$.
\end{theorem}
Thanks to Theorem 3.4 of \cite{buckda}, the following corollary is deduced immediately.
\begin{cor}\label{c0}
Under Assumption \ref{assum}, the unique solution $X_t \in  \mathcal{L}^{2,*}$ of SDE (\ref{equa1}) satisfies 
\begin{equation}\label{expsol}
e_{t,X} X_t(\mathcal{T}_t)\mathcal{E}_t^{-1}(\mathcal{T}_t)  = x+ 
\int_0^t  \tilde{\beta}(u,{\rho}_{{X}_u})\mathcal{E}_u^{-1}(\mathcal{T}_u) du + \int_0^t  \tilde{a}(u,{\Gamma}_{X_u}) \mathcal{E}_u^{-1}(\mathcal{T}_u) dW_u^H.
\end{equation}
\end{cor}
Therefore, the unique solution of (\ref{equa1}), for every $0 \leq t \leq T$, is of the form
\begin{equation}\label{solxt}
X_t(\mathcal{T}_t)= e_{t,X}^{-1}\mathcal{E}_t(\mathcal{T}_t) U_t^{(0)}, 
\end{equation}
where $U_t^{(0)}$ is denoted the rightside of Equation (\ref{expsol}).
This representation show that $X_t$ is Malliavin differentiable, as the processes $\mathcal{E}_t$ and $U_t^{(0)}(\mathcal{A}_t)$ are differentiable.  
\begin{example}\label{example}(Linear fractional mean-field SDEs)
Let the stochastic process $\{Y_{t}\}_{t\in [0,T]}$ be the solution of the linear fractional mean-field SDE
\begin{equation}\label{lin}
dY_{t}=b(t,\rho_{Y_t})Y_{t}dt +C_t Y_{t}dW_{t}^{H}, \qquad t \in [0,T], \qquad Y_{0}=1,
\end{equation}
According to (\em{\ref{expsol}}), the unique solution $Y_t$ satisfies
\begin{equation}\label{yy}
exp\Big\{-\int_0^t b(s,\rho_{{Y}_s}) ds\Big\} Y_t(\mathcal{T}_t)=  \mathcal{E}_t(\mathcal{T}_t).
\end{equation}
However, Theorem \em{2.5} of \cite{flinear}, thanks to  Proposition \em{\ref{ito}}, shows that this solution is
$$Y_t= exp\Big\{\int_0^t C_s dW^H_s -\frac12 \int_0^t\int_0^t \phi(s,s')C_s C_{s'}dsds'+\int_0^t b(s,\rho_{Y_s}^{1}) ds\Big\},$$ 
which coincides with the expression (\em{\ref{yy}}), due to 
$$\langle K_H^* C1_{[0,t]} , K_H^* C1_{[0,t]}\rangle_{L^2[0,T]} = \langle 1_{[0,t]},1_{[0,t]}\rangle_{\mathcal{H}},$$
that results $$\mathcal{E}_t = exp\Big\{\int_0^t C_s dW^H_s -\frac12 \int_0^t\int_0^t \phi(s,s')C_sC_{s'}dsds'\Big\}. $$
The authors have also shown that $D^H_r Y_t = Y_t C_r1_{0 \leq r \leq t}$, equivalently $D_r^\phi Y_t=Y_t \int_0^t \phi(r,s)C_sds$ (for detailes, see Lemma \em{2.2} of \cite{flinear} and the relation (\ref{relation1})). This expression, which shows that the drift coefficient of SDE (\ref{lin}) does not appear in $D^HY_t$ directly, in connection with Ito formula lead to conclude that the process $Y_t^{-1}$ satisfying
\begin{align}\label{linconver}
 dY^{-1}_{t}&=\Big\{- b(t,\rho_{Y_t}^{1})+2C_t\int_{0}^{t}\Phi(s,t) C_sds\Big\}Y^{-1}_{t}dt - C_tY^{-1}_{t}dW^{H}_{t} \quad t\in[0,T],\nonumber \\
 Y_0^{-1}&=1,
 \end{align}
is also differentiable and $D^H_r Y_t^{-1}=-Y_t^{-1} C_r1_{0 \leq r \leq t}$ and in other sense $D^\phi_rY_t^{-1}=-Y_t^{-1} \int_0^t \phi(r,s)C_s ds$. 
\end{example}
Let us finish this part by giving a transformation to omit the term $a(.,.)$ in the SDE (\ref{equa1}) and make a new stochastic process $K_t$ which are needed in sequence. To this end, we need an additional condition in order to have a unique solutionfor new constructed SDE.\\
{\bf H}. Assume that $\frac{d}{dt}\Big(e_{t,X} a(t,\Gamma_{X_t})/{C_t}\Big)$ satisfy  the parts 3 and 4 of Assumption \ref{assum}.\\
Apply Ito formula (Proposition \ref{ito}) for the process $K_{t}=X_{t}+ a(t,\Gamma_{X_t})/{C_t}$ to deduce 
\begin{align*}
dK_{t}&=dX_{t}+\frac{\partial}{\partial t}\Big(\frac{\gt}{C_{t}}\Big) dt\\
&=\Big\{b(t,\rho_{X_t})X_{t}+\beta(t,\rho_{X_t})\Big\}dt+\frac{\partial}{\partial t} \Big(\frac{ a(t,\Gamma_{X_t})}{C_t} \Big)dt+ \Big\{C_{t}X_{t}+\gt  \Big\}dW_{t}^{H}\\
&=\Big\{b(t,\rho_{X_t}) K_{t}+\frac{\partial}{\partial t}\Big(\frac{\gt}{C_{t}}\Big)- b(t,\rho_{X_t})\frac{\gt}{C_{t}}+\ft \Big\}dt+C_{t}K_{t}dW_{t}^{H}\\
&=:  \Big\{b(t,\rho_{X_t}) K_{t}+ \beta_K(t, \rho_{K_t}, \Gamma_{K_t}) \Big\} dt+ C_{t}K_{t}dW_{t}^{H}.
\end{align*}
 Under condition ({\bf H}), Corollary \ref{c0} and the relation (\ref{expsol}) yield that  
\begin{align}
e_{t,X} K_t  (\mathcal{T}_t)&=\mathcal{E}_{t}(\mathcal{T}_t)\Big(K_0+\int_{0}^{t}\Big\{e_{s,X}\beta_{K}(s, \rho_{K_s}, \Gamma_{K_s}) \Big\}\mathcal{E}_{s}^{-1}(\mathcal{T}_s)ds\Big),  \label{ag2}
\end{align}
and since for every $s \leq t$, $\mathcal{E}_{s}^{-1}\Big(\mathcal{T}_s(\mathcal{A}_t)\Big)=\mathcal{E}_{s}^{-1}$, it implies that  
\begin{equation}\label{kt}
e_{t,X} K_t=\mathcal{E}_{t}\Big(K_0+\int_{0}^{t}\Big\{e_{s,X}\beta_{K}(s, \rho_{K_s}, \Gamma_{K_s}) \Big\}\mathcal{E}_{s}^{-1} ds\Big)=:\mathcal{E}_{t} V_t, 
\end{equation}
and 
\begin{equation}\label{vt}
\sup_{0 \leq t \leq T}\E\Big(V_t^2\Big) = \sup_{0 \leq t \leq T}\E\Big(e_t^2 K_t^2(\mathcal{T}_t)\mathcal{E}_{t}^{-2}(\mathcal{T}_t)\Big) = \sup_{0 \leq t \leq T}\E\Big(e_t^2 K_t^2  \mathcal{E}_{t}^{-1}\Big) < \infty.
\end{equation}
\section{Malliavin's derivative of $X_t$}
 We will give an expression for the Malliavin derivative of the solution $X$  and obtain a relation between the integral of its derivative and $\mathcal{E}$ in this section. To do this, according to the definition of $K_t$, it is sufficient to find an exposition of $D^HK_t$ as a solution of some stochastic differential equation. \\
Taking the derivative of both sides of (\ref{kt}) and substituting the derivative of $\mathcal{E}^{-1}_{s}$ from Example \ref{example} in it to reach the following expression for the Malliavin derivative of $K_t$.
\begin{align}
e_{t,X} D_r^H K_t  &=  \mathcal{E}_{t}C_r 1_{[0,t]}(r) V_t - \mathcal{E}_{t}C_r 1_{[0,t]}(r) \Big(\int_{r}^{t}\Big\{e_{s,X}\beta_{K}(s, \rho_{K_s}, \Gamma_{K_s}) \Big\}\mathcal{E}_{s}^{-1} ds\Big) \nonumber\\
&= \mathcal{E}_{t}C_r 1_{[0,t]}(r) \Big(K_0+\int_{0}^{r}\Big\{e_{s,X}\beta_{K}(s, \rho_{K_s}, \Gamma_{K_s}) \Big\}\mathcal{E}_{s}^{-1} ds\Big)\nonumber\\
&=  \mathcal{E}_{t}C_r 1_{[0,t]}(r) e_{r,X} K_r \mathcal{E}^{-1}_{r}. \label{Dkt}
\end{align}
In view of Corollary \ref{c0}, we conclude the form of stochastic differential equation associated with the weak derivative $D^HK_t$.
\begin{prop}\label{MK}
Under Assumption \ref{assum} and {\bf H}, the Malliavin derivative of $K_t$ satisfies the following equation 
\begin{equation}\label{dhk} 
D_r^{H}K_{t}=C_rK_{r}1_{[0,t]}(r)+\int_r^t  b(s,\rho_{X_s}) D_r^{H}K_{s} ds+\int_r^t C_s D_r^{H}K_{s}dW^{H}_{s}, 
\end{equation}
for every $0 \leq r \leq t$. In addition, for every $h \in \mathcal{H}$ 
\begin{align*} 
<D^{H}K_{t},h>_{\mathcal{H}}&=<C_.K_{.}1_{[0,t]},h>_{\mathcal{H}}+\int_0^t  b(s,\rho_{X_s}) <D^{H}K_{s},h>_{\mathcal{H}} ds\\
&+\int_0^t C_s <D^{H}K_{s},h>_{\mathcal{H}}dW^{H}_{s}.
\end{align*}
\end{prop}
\begin{rem}
For $h=1_{[0,t]}$, in view of (\ref{relation1}), the process $D^{\phi}_{r}K_{t}$ is the solution of the following SDE.
\begin{equation*} 
D^{\phi}_{r}K_{t}=\int_0^t  b(s,\rho_{X_s}) D^{\phi}_{r}K_{s} ds+\int_0^t  C_u K_{u}\phi(r,u)du+\int_0^t C(s) D^{\phi}_{r}K_{s}dW^{H}_{s}, 
\end{equation*}
with $D^{\phi}_{r}K_{0}=0$ and for every $r \geq 0$.
\end{rem}
Now, we are ready to establish a relation, that demonstrates the process $\mathcal{E}$ as an inner product of the Malliavin derivative with other Skorokhod integral processes needed in the Bismut formula.  
\begin{rem}
 If  $\beta_{K}(s, \rho_{K_s}, \Gamma_{K_s}) = 0$ for every $s \in [0,T]$, especially if $a(.,.)=\beta(.,.)=0$, then for $K_0 \neq 0$ we have
\begin{equation}\label{espe}
\int_0^T  \frac{1}{K_0 T}C_r^{-1} 1_{[0,t]}(r) D_r^H K_t dr = e_{t,X}^{-1}\mathcal{E}_{t}.  
\end{equation}
Also, if  $K_0 \beta_{K}(s, \rho_{K_s}, \Gamma_{K_s}) > 0$, for every $s \in [0,T]$, then  
\begin{equation*}
\int_0^T \frac{1}{T} V_r^{-1}C_r^{-1} 1_{[0,t]}(r) D_r^H K_t dr = e_{t,X}^{-1}\mathcal{E}_{t}.  
\end{equation*}
\end{rem}
Nevertheless, in the case $\beta_{K}(s, \rho_{K_s}, \Gamma_{K_s})\neq 0$, the process $K_r$ would not be invertible almost everywhere, in general, then the process $e_t^{-1}\mathcal{E}_{t}$ can not be obtain with respect to the derivative of $K_t$ by taking inversion of (\ref{Dkt}) and integrate of that with respect to $r$. So, we take a new trick to reach this goal in the sequence of this section. \\
Let 
$$S= \inf \Big\{t >0; \int_0^t 1_{\{V_s \neq 0\}} ds >0\Big\} \wedge T.$$
It follows that for every $S \leq t$, there exists a set $A \subset (0,t]$ of positive Lebesque measure such that  for any $s \in A$,  $$V_s > \lambda(s)>0.$$
In sequence, in addition to Assumption \ref{assum}, we need also the following assumption.
\begin{assum}\label{assum2}
The function $C$ is invertible and $C^{-1}$ is H\"{o}lder continuous of order $\alpha_0 \in (H-\frac12, 1]$; 
\begin{equation*}
\vert C^{-1}_t -C^{-1}_s \vert \leq K \vert t-s \vert^{\alpha_0}, \quad t,s \in [0,T]
\end{equation*}
\end{assum}
Now, the following main result will establish on the set $\{S < t\}$. 
\begin{theorem}\label{vaznn}
Under Assumptions {\em \ref{assum}}, {\em \ref{assum2}} and {\bf H}, for every $0 < t \leq T$, on the set  $\{S < t\}$, there exists some random variable $V_t^{0} \in \mathcal{H}\cap  Dom(K_H^*)$ satisfying 
\begin{equation*}
\langle  D^H K_t, V_t^{0} \rangle_\mathcal{H} = e_{t,X}^{-1}\mathcal{E}_{t},
\end{equation*}
and $ K_H^* V_T^0  \in Dom(\delta_W)$.
\end{theorem}
\begin{proof}
According to the assumption  $\{S < t\}$, there exists a set $A \subset (0,t]$ of positive Lebesque measure such that  for any $s \in A$,  $$V_s \geq  \lambda(s),$$
where $\lambda(s) >0$. Therefore, 
$$\int_0^t V_s^2 ds \geq \int_A \lambda(s) ds =:k >0.$$   
Multiply both sides of (\ref{Dkt}) in 
$V_r$ to deduce
\begin{align}
\int_0^T  C_r^{-1} 1_{[0,t]}(r) D_r^H K_t V_r dr &= e_{t,X}^{-1}\mathcal{E}_{t} \Big( \int_0^t V_r^2 dr \Big).\label{multiplyv}
\end{align}
Now, define 
$$V_{2,t}(r) =   C_r^{-1} 1_{[0,t]}(r)\frac{V_r}{ \int_0^t V_r^2 dr}.$$
From Assumption \ref{assum2}, we see that $V_{2,t} \in L^2[0,T]$. Also, it is easily verified that $V_{2,t} \in I_0^{H+\frac12}(L^2[0,T]; \mathbb{R})$. Thus there exists some random variable $V_t^{0} \in Dom(K_H^*)$  such that $K_H o K_H^* V_t^0 = \int_0^. V_{2,t}(r)dr$.  To complete the proof, it is sufficient to show that $K_H^{-1}(\int_0^. V_{2,t}(r) dr) \in Dom(\delta_W)$. Since ${D}_W^{1,2} \subset Dom(\delta_W)$ and $K_H^* V_t^0$ is an adapted process, (see Proposition 1.3.11 in \cite{nualart}), we will show that it belongs to $L^2([0,T] \times \Omega)$. To this end, as $t \neq T$, for some constant $c$, we derive
\begin{align*}
 K_H^{-1}\Big(\int_0^. V_{2,t}(r) dr\Big)(u)&1_{t < u \leq T} \int_0^t  V_s^2  ds\\
& = (H-\frac12)u^{H-\frac12}1_{t < u \leq T}\int_0^u \frac{ - s^{\frac12-H}1_{[0,t]}(s)C_s^{-1} V_s }{(u-s)^{H+\frac12}} ds\\
& \leq  (H-\frac12)u^{H-\frac12}1_{t < u \leq T}\Big(\int_0^u \frac{  s^{1-2H} }{(u-s)^{2H+1}} ds\Big)^\frac12\Big( \int_0^t  C_s^{-2} V_s^2  ds\Big)^\frac12\\
& \leq  (H-\frac12) \frac{u^{H-\frac12}}{(u-t)^{H+\frac12}} 1_{t < u \leq T}\Big(\int_0^u  s^{1-2H} ds\Big)^\frac12\Big( \int_0^t  C_s^{-2} V_s^2  ds\Big)^\frac12\\
& \leq  c\frac12 (\frac{u}{u-t})^{H+\frac12}u^{-1-H} 1_{t < u \leq T}
\Big(\int_0^t  V_s^2  ds\Big)^\frac12,
\end{align*}
which implies that 
\begin{align*}
\E\Big( \int_t^T 1_{S <t} \Big[K_H^{-1}\Big(\int_0^. V_{2,t}(r) dr\Big)(u) \Big]^2 du \Big) \leq \frac{c}{\sqrt{k}}\E\Big(  \int_t^T (\frac{u}{u-t})^{2H+1}u^{-2-2H} du  \Big) < \infty.
\end{align*}
Now, if $0 \leq u \leq t$, which also includes the case $t=T$, we know that
\begin{align*}
 K_H^{-1}\Big(\int_0^. V_{2,t}(r) dr\Big)(u)&1_{[0,t]}(u)\int_0^t  V_s^2  ds\\
& = \frac{1}{\Gamma(\frac32-H)}\Big(u^{\frac12-H}1_{[0,t]}(u)C_u^{-1}  V_u\\
& + (H-\frac12) C_u^{-1} V_u 1_{[0,t]}(u) \int_0^u \frac{(1_{[0,t]}(u)-1_{[0,t]}(s))}{(u-s)^{H+\frac12}} ds\\
& + (H-\frac12)u^{H-\frac12} C_u^{-1} V_u 1_{[0,t]}(u)\int_0^u \frac{u^{\frac12-H}-s^{\frac12-H}}{(u-s)^{H+\frac12}}1_{[0,t]}(s) ds\\
& + (H-\frac12)u^{H-\frac12} V_u 1_{[0,t]}(u) \int_0^u \frac{s^{\frac12-H}1_{[0,t]}(s)(C_u^{-1}-C^{-1}_s)  }{(u-s)^{H+\frac12}} ds\\
& + (H-\frac12)u^{H-\frac12}1_{[0,t]}(u) \int_0^u \frac{s^{\frac12-H}1_{[0,t]}(s)C^{-1}_s ( V_u - V_s )}{(u-s)^{H+\frac12}} ds\Big)\\
&=: \frac{1}{\Gamma(\frac32-H)}\Big( I_1 + I_2 + I_3+I_4+I_5\Big).
\end{align*}
Since $0 \leq u \leq t$, 
\begin{equation}\label{finit}
u^{2H-1}\int_0^u \frac{u^{\frac12-H}-s^{\frac12-H}}{(u-s)^{H+\frac12}}1_{[0,t]}(s) ds = \int_0^1 \frac{r^{\frac12 -H}-1}{(1-r)^{\frac12+H}} dr =:g_0 < \infty,
\end{equation}
and 
\begin{equation*}
\int_0^u \frac{(1_{[0,t]}(u)-1_{[0,t]}(s))}{(u-s)^{H+\frac12}} ds =0.
\end{equation*}
From Assumption \ref{assum2}, it follows that for some constant $c$
\begin{equation*}
 u^{2H-1} 1_{[0,t]}(u)\Big(\int_0^u \frac{s^{\frac12-H}1_{[0,t]}(s)(C_u^{-1}-C^{-1}_s)  }{(u-s)^{H+\frac12}} ds \Big)^2 = cu^{2\alpha_0-2H+1} \leq ct^{2\alpha_0-2H+1}.
\end{equation*}
These imply that 
\begin{equation*}
\E\Big(\int_0^T  1_{S <t}\frac{ I_2^2(u)+ I_3^2(u) +I_4^2(u)}{\Big(\int_0^t  V_s^2  ds\Big)^2}du\Big) < \infty.
\end{equation*}
Additionally, from (\ref{vt})
\begin{equation*}
\E\Big(\int_0^T  1_{S <t}\frac{ I_1^2(u)}{\Big(\int_0^t  V_s^2  ds\Big)^2} du \Big) \leq \sup_{0 \leq s \leq t}\E\Big(1_{S <t}\frac{ V_s^2}{\Big(\int_0^t  V_s^2  ds\Big)^2} \Big)\int_0^t u^{1-2H} du < \infty.
\end{equation*}
Now, to bound the last term $I_5$, we apply the definition of $V_t$ and H\"{o}lder inequality. Therefore, for some constant $c$,  
\begin{align*}
\Big(\int_0^u \frac{s^{\frac12-H}1_{[0,t]}(s)C^{-1}_s ( V_u - V_s )}{(u-s)^{H+\frac12}} ds \Big)^2 & \leq c\int_0^u \frac{s^{1-2H}}{(u-s)^{2H}} ( \int_s^u \mathcal{E}_v^{-2} dv) ds  \\
& \leq c \sup_{0 \leq v \leq t} \mathcal{E}_v^{-2}\int_0^u \frac{s^{1-2H}}{(u-s)^{2H-1}} ds \\
& \leq \frac{2c}{2-2H} u^{2-2H}\sup_{0 \leq v \leq t} \mathcal{E}_v^{-2}.
\end{align*}
which yields finally 
\begin{equation*}
\E\Big(\int_0^T  1_{S <t}\frac{ I_5^2(u)}{\Big(\int_0^t  V_s^2  ds\Big)^2} du \Big) \leq \frac{c}{2} t^2\E\Big(1_{S <t}\frac{\sup_{0 \leq s \leq t}  \mathcal{E}_s^{-2}}{\Big(\int_0^t V_s^2  ds\Big)^2} \Big)< \infty,
\end{equation*}
for some another constant $c>0$.
These facts together, show that $K_H^{-1}\Big(\int_0^. V_{2,t}(r) dr\Big) \in Dom(\delta_W)$ and the proof is comoleted.
\end{proof}
\begin{rem}
In especial case, Remark 3.1, the stochastic process $V_{2,t}(r)$ will be in the form 
  $$V_{2,t}(r) =   \frac{1}{T}V_r^{-1} C_r^{-1} 1_{[0,t]}(r),$$ 
Obviously, this process belongs to $ I_0^{H+\frac12}(L^2[0,T]; \mathbb{R})$ and also $K_H^{-1}(\int_0^. C_r^{-1} 1_{[0,t]}(r) dr) \in Dom\delta_W \cap  \mathcal{H}$. 
\end{rem}
\section{Stochastic Flow of The Equation}
This section is devoted to proving the existence of the stochastic flow of the solution $K_t$ and establishing a relationship between this flow and the Malliavin derivative of $K_t$, which will be essential in the remainder of the paper.\\
Denote by $K_t^u$( and $X_t^x$ ) the solution of (\ref{equa2})( and (\ref{equa1})) with initial condition $u$( and $x$). Also, let  
$R(t,\rho_{X_t}) := \partial_2 b(t,\rho_{X^x_t}) \frac{\partial \rho_{X^x_t}}{\partial x}$ and $r(t, f_{X_t}):=\frac{\partial }{\partial x}\beta_K(t,\rho_{K_t}, \Gamma_{K_t})$, where $\partial_{2}b(.,.)$ is the derivative of $b$ with respect to the second component.\\
We need some additional conditions to show the existence of the stochastic flow of the solution. 
\begin{assum}\label{assum3}
i) The function $e_{.,.} r(.,.)$ is a bounded measurable function and for every $t,s \in [0, T]$ and stochastic processes $X$ and $Z$, we have 
\begin{align*}
\vert e_{t,X}r(t,f_{X_t})-e_{s,Z}r(s,f_{Z_s}) \vert \leq K\Big(\vert t-s\vert+\E\Big\vert e_{t,X}{X_t}-e_{s,Z}{Z_s} \Big\vert\Big),
\end{align*}
ii)  The function $\frac{\partial }{\partial x}b(t,\rho_{X_t})$ is a bounded measurable function and for every $t,s \in [0,T]$ and stochastic processes $X$ and $Z$, we have 
\begin{align*}
\vert R(t,\rho_{X_t})-R(s,\rho_{Z_s}) \vert \leq K'\vert t-s\vert^{\alpha_1},
\end{align*} 
where $K'$ is a constant and $\alpha_1 \in (H-\frac12, 1]$.\\
iii)   The function $e_{.,.} \frac{\partial }{\partial x}a(.,.)$  is a bounded measurable function.
\end{assum}
\begin{theorem}
Under conditions (i) and (ii) of Assumption \ref{assum3}, the function $u \rightarrow K_t^u$ is contiuously differentiable and $K_t^{K_0}$ admits a stochastic flow satisfying the following SDE:
\begin{align}
d\frac{\partial K_{t}}{\partial x}& = \Big\{ b(t,\rho_{X_t^{x}}) \frac{\partial K_{t}}{\partial x}+R(t,\rho_{X_t^{x}})K_{t} + r(t,f_{X_t^{x}})\Big\}dt+C_t \frac{\partial K_{t}}{\partial{x}}dW^{H}_{t}, \label{flow}
\end{align}
with the initial condition $\frac{\partial}{\partial x}K_0=1+\frac{1}{C_0}\frac{\partial a(0,\psi(x))}{\partial x}\psi'(x)=: k_0$, 
\end{theorem}
\begin{proof}
It follows from (\ref{kt}), part (i) of Assumption \ref{assum3} and H\"older inequality that 
\begin{align*}
\vert e_{t,X^{x+h}} K_t^{K_0+h}-e_{t,X^x} K_t^{K_0}\vert^2 & \leq \mathcal{E}_{t}^2\Big(h+\int_{0}^{t}\vert e_{s,X^{x+h}}r(s, f_{X_s^{x+h}})-e_{s,X^{x}}r(s, \rho_{X_s^{x}})\vert \mathcal{E}_{s}^{-1} ds\Big)^2\\
& \leq  \mathcal{E}_{t}^2 h^2 +  \mathcal{E}_{t}^2  \Big(\int_0^t \mathcal{E}_{s}^{-1} ds\Big)\int_{0}^{t} \E\Big(\vert e_{s,X^{x+h}}{X_s^{x+h}}-e_{s,X^{x}}{X_s^{x}}\vert^2 \mathcal{E}_{s}^{-1}\Big)  ds\\
& \leq  \mathcal{E}_{t}^2(1+c_0t \int_0^t \mathcal{E}_{s}^{-1} ds) h^2 ,
\end{align*}
where we used (\ref{xz}) in the last inequality. Therefore, the function $x \rightarrow  e_{t,X^{x}} K_t^{K_0}$ is differentiable. This result in connection with differentiability of $b(.,.)$ deduce that the function $u \rightarrow K_t^u$ is continuously differentiable.\\
Now, take the derivative with respect to x from (\ref{kt}) to derive 
\begin{align}
e_{t,X} \frac{\partial K_{t}}{\partial x}+ e_{t,X} K_t \int_0^t R(s,\rho_{X_s^{x}})ds &= \mathcal{E}_{t}\Big[\frac{\partial}{\partial x}K_0+ \int_0^t e_{s,X} r(s,f_{X_s^{x}})\mathcal{E}_{s}^{-1}  ds\nonumber\\
&+ \int_0^t  e_{s,X}\Big(\int_0^s R(u,\rho_{X_u^{x}})du\Big) \beta_K(s,\rho_{K_s}, \Gamma_{K_s})  \mathcal{E}_{s}^{-1}  ds\Big]. \label{rond1}
\end{align}
Integration by parts formula shows that 
\begin{align}
\int_0^t e_{s,X}R(s,\rho_{X_s^{x}})K_s \mathcal{E}_{s}^{-1}  ds & =  e_{t,X} K_t \mathcal{E}_{t}^{-1} \int_0^t R(s,\rho_{X_s^{x}})ds \nonumber\\
&-  \int_0^t  e_{s,X}\Big(\int_0^s R(u,\rho_{X_u^{x}})du\Big) \beta_K(s,\rho_{K_s}, \Gamma_{K_s})  \mathcal{E}_{s}^{-1}  ds. \label{rond2}
\end{align}
Substituting (\ref{rond2}) into (\ref{rond1}) results that 
\begin{align*}
e_{t,X} \frac{\partial K_{t}}{\partial x} &= \mathcal{E}_{t}\Big[\frac{\partial}{\partial x}K_0+ \int_0^t e_{s,X} r(s,f_{X_s^{x}})\mathcal{E}_{s}^{-1}  ds+\int_0^t e_{s,X}R(s,\rho_{X_s^{x}})K_s \mathcal{E}_{s}^{-1}  ds\Big].
\end{align*}
Similar argument of obtaining $K_t$, and according to Corollary \ref{c0} in connection with Assumption \ref{assum3}, state that 
$K_t$ is a solution of SDE (\ref{flow}).
\end{proof}
To find a relationship between the flow of $K_t$ and its Malliavin derivative, let us define the process $\mathcal{M}^h_t$ for every $h\in \mathcal{H}$ as the solution of the following SDE 
\begin{equation}\label{mh} 
d\mathcal{M}^h_t =  b(t,\rho_{X_t^{x}}) \mathcal{M}^h_t dt+ r(t,f_{X_t^{x}}) dt + C_t \mathcal{M}^h_t dW_t^H , \qquad \mathcal{M}^h_0=k_0.
\end{equation}
According to Theorem \ref{c2} and Corollary \ref{c0}, this equation has a unique solution of the form
$$\mathcal{M}^h_t = e_{t,X} ^{-1}\mathcal{E}_t \Big(k_0+\int_0^t r(s)e_{s,X} \mathcal{E}_s^{-1} ds\Big),$$
\begin{theorem}\label{mt}
Under Assumptions \em{\ref{assum}, \ref{assum2}}, \ref{assum3} and {\bf H}, there exists some  $h\in\mathcal{H} \cap Dom(\delta_W)$ such that 
$$\frac{\partial  K_{t}}{\partial x}-< D^{H}K_t,h >_{\mathcal{H}} =
\mathcal{M}^h_t.$$ 
\end{theorem}
\begin{proof}
ۤFollowing a similar argument of Proposition 3.1 in \cite{fan}, we verify that 
$$< C_.K_. 1_{[0,t]},h >_{\mathcal{H}} = \int_0^t C_s K_s d(R_Hh)(s), \quad {for ~all} ~h \in \mathcal{H}.$$
Take $\mathcal{N}^h_t := \frac{\partial  K_{t}}{\partial x}-< D^{H}K_t,h >_{\mathcal{H}}$ and use Ito formula to derive 
\begin{align*}
d\mathcal{N}^h_t  =b(t,\rho_{X_t^{x}}) \mathcal{N}^h_t dt + R(t,\rho_{X_t^{x}})K_t dt  - C(t) K_t d(R_Hh)(t) + r(t,f_{X_t^{x}}) dt + C_t \mathcal{N}^h_t dW_t^H.
\end{align*}
The function $C_.^{-1}R(.,.) \in I_0^{H+\frac12} L^2([0,T];\mathbb{R})$. So, one can choose a function $h \in \mathcal{H}$ satisfying 
\begin{equation}\label{h1}
R_Hh(t)=\int_0^t C_s^{-1}R(s,\rho_{X_s^{x}})ds.
\end{equation}
In this case, from the uniqueness of the solution of (\ref{mh}), we conclude that $ \mathcal{N}^h_t = \mathcal{M}^h_t $ in $L^{2,*}([0,T]\times \Omega)$. To complete the assertion, it is sufficient to show that $K_H^{-1}( R_Hh) \in Dom(\delta_W)$. The definition of $K_H^{-1}$ follows that  
\begin{align*}
 K_H^{-1}\Big(\int_0^. C_r^{-1}R(r,\rho_{X_r^{x}}) dr\Big)(u)
& = \frac{1}{\Gamma(\frac32-H)}\Big(u^{\frac12-H}C_u^{-1}  R(u,\rho_{X_u^{x}}) \\
& + (H-\frac12)u^{H-\frac12} C_u^{-1} R(u,\rho_{X_u^{x}}) \int_0^u \frac{u^{\frac12-H}-s^{\frac12-H}}{(u-s)^{H+\frac12}} ds\\
& + (H-\frac12)u^{H-\frac12}  R(u,\rho_{X_u^{x}})  \int_0^u \frac{s^{\frac12-H}(C_u^{-1}-C^{-1}_s)  }{(u-s)^{H+\frac12}} ds\\
& + (H-\frac12)u^{H-\frac12} \int_0^u \frac{s^{\frac12-H}C^{-1}_s ( R(u,\rho_{X_u^{x}}) - R(s,\rho_{X_s^{x}}) )}{(u-s)^{H+\frac12}} ds\Big).
\end{align*}
Similar to the proof of Theorem \ref{vaznn} and notice to Assumption \ref{assum3}, we verify that 
$$\E\Big(\int_0^T  \Big(K_H^{-1}\Big(\int_0^. C_r^{-1} R(r,\rho_{X_r^{x}}) dr\Big)(u)\Big)^2 du \Big) < \infty.$$
which completes the proof.
\end{proof}
Now, define the stochastic  process $$\mathcal{W_{HT}}(t):= \frac{\mathcal{M}^h_t}{<D^{H}K_t, V_t^0>_{\mathcal{H}}},$$ where $V_t^0 \in \mathcal{H}\cap Dom(K^*_{H})$, defined in Theorem \ref{vaznn}. 
\begin{cor}
For the stochastic process  $V_t^0 \in \mathcal{H}\cap Dom(K^*_{H})$, we have 
$$\mathcal{W_{HT}}(t)= k_0+\int_0^t  e_{s,X} r(s,f_{X_s^{x}}) \mathcal{E}_s^{-1}ds \in Dom(\delta_W).$$
\end{cor}
\section{Bismut-Elworthy-Li formula}
In this section, we find a Bismut-Elworthy-Li type formula with the help of Theorems \ref{vaznn} and \ref{mt}. 
The following lemma prepares some requirements to reach this main result.

\begin{lemma}\label{deltahkh}
Assume that $F, G \in D_W^{1,2} \cap D^{1,2}$, and continuously differentiable bounded function $\phi$ is given. For every $v \in \mathcal{H}\cap Dom(K^*_{H})$ such that $<D^{H}F, v>_{\mathcal{H}} \neq 0$ and every random variable $G$ such that $K_H^*v \frac{G}{<D^{H}F, v>_{\mathcal{H}}} \in Dom(\delta_W)$, it holds that  
\begin{equation*}
\mathbb{E}\Big(\phi'(F)G\Big) = \mathbb{E}\left(\phi(F) \delta^H\Big((K_H^*)^{-1} \frac{GK_H^*v}{<D^{H}F, v>_{\mathcal{H}}}\Big)\right).
\end{equation*}
\end{lemma}
\begin{proof}
Due to Proposition \ref{relationdeltaH} and Proposition \ref{relationDH}, for $u=K_H^*v$ we derive
\begin{align*}
\mathbb{E}\Big(\phi'(F)G\Big) & = \mathbb{E}\left(\phi(F) \delta\Big( \frac{Gu}{<DF,u>_{L^2[0,T]}}\Big)\right)\\
& = \mathbb{E}\left(\phi(F) \delta\Big( \frac{GK_H^*v}{<K_H^* D^H F, K_H^*v>_{L^2[0,T]}}\Big)\right)\\
& = \mathbb{E}\left(\phi(F) \delta\Big( \frac{GK_H^*v}{< D^H F,v>_{\mathcal{H}}}\Big)\right)\\
& = \mathbb{E}\left(\phi(F) \delta^H\Big((K_H^*)^{-1} \frac{GK_H^*v}{< D^H F, v>_{\mathcal{H}}}\Big)\right).
\end{align*}
\end{proof}
\begin{theorem}\textbf{(Bismut-Elworthy-Li formula I)}\label{Bismut-Elworthy-Li formula-K-t}
 Under Assumption \ref{assum3}, for every measurabe function $\phi:\mathbb{R}^{d}\longrightarrow \mathbb{R}$ such that $\phi(K_{t}^{K_0})\in {L}^{2}(\Omega)$ and $\phi \in C_b^1(\mathbb{R})$ (the space of bounded and continuously differentiable functions on $\mathbb{R}$), sensitivity of the payoff $\phi$ with respect to the initial value $x$ is given by the following equation. 
\begin{align*}
\frac{\partial}{\partial x}\mathbb{E}\left(\phi(K_{t}^{K_0})\right)&=\mathbb{E}\left(\phi(K_{t}^{K_0}) \delta^H\Big(1_{S \leq t}h+(K_H^*)^{-1}\mathcal{W_{HT}}(t) 1_{S \leq t}K_H^{-1}w_t \Big)\right)\\
&= \mathbb{E}\left(\phi(K_{t}^{K_0}) \delta \Big(1_{S \leq t}K_H^{-1}\{\int_0^. C_s^{-1}1_{[0,t]}(s) R(s,\rho_{X_s^{x}}) ds \}\Big)\right)\\
&+\mathbb{E}\left(\phi(K_{t}^{K_0}) \delta \Big(1_{S \leq t}\mathcal{W_{HT}}(t) K_H^{-1}w_t \Big)\right),
\end{align*}
 for $h$ satisfying (\ref{h1}). Especially, as $\beta_K(t,\rho_{K_t}, \Gamma_{K_t} )=0$,
\begin{equation*}
\frac{\partial}{\partial x}\mathbb{E}\Big(\phi(K_{t}^{K_0})\Big)= \mathbb{E}\left(\phi(K_{t}^{K_0}) \delta \Big(1_{S \leq t}K_H^{-1}\Big\{\int_0^. C_s^{-1}1_{[0,t]}(s)[R(s,\rho_{X_s^{x}})+\frac{1}{K_0T}] ds \Big\}\Big)\right).
\end{equation*}
\end{theorem}
\begin{proof}
For simplicity we use $K_{t}$ instead of $K_{t}^{K_0}$ and note that on the set $\{S>t\}$, the process $K_t$ is zero. From Theorem \ref{mt}, Lemma \ref{deltahkh} and Theorem \ref{vaznn}, for $V_t^0 \in \mathcal{H}\cap Dom(K^*_{H})$, we result
\begin{align*}
\frac{\partial}{\partial x}\mathbb{E}\Big(\phi(K_{t})\Big) & =\mathbb{E}\Big(1_{S > t}\phi'(K_{t})\frac{\partial K_t}{\partial x}\Big)+\mathbb{E}\Big(1_{S \leq t}\phi'(K_{t})\frac{\partial K_t}{\partial x}\Big)\nonumber\\
&=\mathbb{E}\Big(\phi'(K_{t})1_{S \leq t}< D^{H}K_t,h >_{\mathcal{H}}\Big)+ \mathbb{E}\Big(\phi'(K_{t})1_{S \leq t} \mathcal{M}^h_t\Big)\nonumber\\
&= \mathbb{E}\Big(\phi(K_{t}) \delta^H(1_{S \leq t}h)\Big)+ \mathbb{E}\Big(\phi(K_{t})\delta^H\Big((K_H^*)^{-1}\frac{1_{S \leq t}\mathcal{M}^h_t K_H^{-1}(\int_0^. V_{2,t}(s) ds)}{<D^{H}K_t, V_t^0>_{\mathcal{H}}}\Big)\Big)\nonumber\\
&=  \mathbb{E}\Big(\phi(K_{t}) \delta^H(1_{S \leq t}h)\Big)+  \mathbb{E}\Big(\phi(K_{t})\delta^H\Big((K_H^*)^{-1}(1_{S \leq t}\mathcal{W_{HT}}(t) K_H^{-1}w_t   )  \Big)    \Big),
\end{align*}
where $w_t = \int_0^. V_{2,t}(s) ds$. 
It implies the first part of the claim.  Especially, from Equation (\ref{espe}), $w_t= \int_0^. \frac{1}{K_0T} C_s^{-1}1_{[0,t]}(s) ds $ and the second part of the claim is held exactly in the same way. 
\end{proof}
We are actually able to rewrite the above theorem for the stochastic process $X$ by considering the relation between two processes $K_t$ and $X_t$. 
\begin{theorem}\textbf{(Bismut-Elworthy-Li formula)}
Assume that the measurabe function $\phi:\mathbb{R}^{d}\longrightarrow \mathbb{R}$ belongs to $C_b^1(\mathbb{R})$ and $\phi(X_{T}^{x})\in L^{2}(\Omega)$. Then under Assumption \ref{assum3},
\begin{align*}
\frac{\partial}{\partial x}\mathbb{E}\Big(\phi(X_{T}^{x})\Big)&=\mathbb{E}\left(\phi(X_{T}^{x}) \delta^H\Big(h+(K_H^*)^{-1}\{U(T)K_H^{-1}w_T\}\Big)\right)\\
&=\mathbb{E}\left(\phi(X_{T}^{x}) \delta \Big(K_H^{-1}\Big\{\int_0^. C_s^{-1}1_{[0,t]}(s) R(s,\rho_{X_s^{x}}) ds \Big\}+U(T) K_H^{-1}w_T \Big)\right).
\end{align*}
where $U(T)=\mathcal{W_{HT}}(T)-\frac{\frac{\partial}{\partial x} a(T, \Gamma_{X^x_T})}{C_T}e_T \mathcal{E}_T^{-1} $ and $h$ satisfying  (\ref{h1}).
\end{theorem}
\begin{proof}
 From Theorem \ref{Bismut-Elworthy-Li formula-K-t} and Lemma \ref{deltahkh} as $G=1$, in connection with the fact $D^HK_t=D^HX_t$,
\begin{align*}
\frac{\partial}{\partial x}\mathbb{E}\Big(\phi(X_{T}^{x})\Big)&=
\mathbb{E}\Big(\phi'(K_{T}^{k_0}-\frac{a(T, \Gamma_{X_T^x})}{C_T})\left[\frac{\partial K_T}{\partial x}- \frac{\frac{\partial}{\partial x}a(T, \Gamma_{X_T^x})}{C_T}\right]\Big)\\
&= \mathbb{E}\Big(\phi({X_T^{x}})  \delta^H(h) \Big)+\mathbb{E}\Big(\phi(X_{T}^{x})  \delta^H\Big((K_H^*)^{-1}\{\mathcal{W_{HT}}(T) K_H^{-1}w_T\}\Big) \Big)\\
&- \mathbb{E}\Big(\phi(X_{T}^{x})  \delta^H\Big((K_H^*)^{-1}\frac{\frac{\partial}{\partial x}a(T, \Gamma_{X_T^x})}{C_T}\frac{K_H^* V_T^0}{<D^{H}K_T, V_T^0>_{\mathcal{H}}}\Big)\\
&= \mathbb{E}\Big(\phi(X_{T}^{x})  \delta^H(h) \Big)+\mathbb{E}\Big(\phi(X_{T}^{x})  \delta^H\Big(U(T) K_H^{-1}w_T\Big)\Big).
\end{align*}
\end{proof}
\section{Application}
\noindent The first fractional volatility model was presented in \cite{comte}, which is based on the fractional Ornstein-Uhlenbeck process with a Hurst parameter $H > 1/2$ to describe the slow flattening of smiles and skews as the time to maturity increases.  The volatility model around the fractional literature has subsequently developed in many articles, for instance, see \cite{cheri, comte2, alos2, bayer}. 
Long memory in the volatility process becomes much
needed in order to justify the long-standing conundrums such as steep volatility smiles in
long-term options and co-movements between implied and realized volatility. (see also \cite{wang2}).
A fractional Ornstein–Uhlenbeck process  is defined as the solution of the SDE
\begin{equation}\label{sigma}
d\sigma_t=( \nu- \kappa \sigma_t )dt + u dW_t^H,
\end{equation}
where $\nu, \kappa$, and $u$ are positive parameters and the stochastic integral is a pathwise Riemann-Stieltjes integral.
Gatheral \cite{gatheral} have considered rough stochastic fractional volatility model on the time interval $[0,T]$, as the volatility $x_t$ is the form $x_t =e^{\sigma_t}$ for $\sigma_t$ satisfying (\ref{sigma}). 
The volatility model is essential for investors in predicting their portfolio \cite{foko}, forming suitable hedging strategies \cite{euch, salmon}, pricing volatility indices and other derivative products \cite{jai}, etc.\\
Analysis of variance and volatility swaps,  a forward contract on realized stock variance and volatility, are developed with the underlying
stochastic volatility models. For example, the authors in \cite{swi} choose a
drift-adjusted version of the Heston model, and in \cite{kim} the authors choose the log-normal SABR model with
fractional stochastic volatility to price variance and volatility swaps. It is reasonable to develop the model as a mean-field stochastic differential equation and consider the sensitivity of the price of swaps, without achieving a closed-form exact formula for the price of the swaps, based on the Malliavin derivative of the volatility. The model considered in the next example is inspired of the Example 4.1 in \cite{banos} in fractional literature.
\begin{example}\label{exam1}
Consider the stochastic volatility dynamics $\{\sigma _{t}\}_{t\in[0,t]}$ of an asset with initial volatility $ x>0 $ in Block-Scholes model with coninuous dividend payments represented by the following SDE
\begin{equation}\label{sig}
d\sigma_{t}=(\mu - q\rho_{t}^{x})\sigma_t dt +\alpha\sigma_t dW_{t}^{H},  \qquad \rho_{t}^{x}=\mathbb{E}(\sigma_{t}^{x}),\quad t\in [0,T],
\end{equation}
where $\mu,q,\alpha\in\mathbb{R^+}$, $\mu>0$. It is worth mention that if the parameter $q$ in (\ref{sig}) would be zero, our model coincides with the models we refereed them. \\
A variance swap is a forward contract on realized stock variance $\sigma^2 _{R}(x)=\frac{1}{T}\int_{0}^{T}\sigma_{s}^{2}ds$, with payoff function $N(\sigma_{R}^{2}-K_{var})$ in which $N$ is the notional amount  and $K_{var}$ is strike of the swap. 
The price of a variance forward contract is the discounted value of the payoff in a risk-neutral world with a risk-less rate $r>0$.
\begin{equation*}
P_{var}(x)=\mathbb{E^{Q}}\Big(e^{-rT}(\sigma_{\mathbb{R}}^{2}(x) - K_{var})\Big),
\end{equation*}
where $Q$ is a minimal martingale measure, i.e.;
\begin{equation*}
\forall t\in [0,T]\quad  {M}_{t}^{x}:=\frac{d\mathbb{Q}}{dp}|_{\mathcal{F}_t}=exp\Big\{\int_{0}^{t}\Theta _{u}^{x}dW_{u}^{H} -\frac{1}{2}\int _{0}^{t}\int_{0}^{t}\Theta_{\nu}^{x}\Theta _{u}^{x}\phi(\nu, u)du d\nu\Big\},
\end{equation*}
for some $\Theta(.)$ satisfying $\mu-q\rho_{t}^{x} - \alpha \int_{0}^{t}\phi (t,u)\Theta_u du =r$, (see also \cite{hab}).\\
To consider the sensitivity of the price with respect to initial volatility $x$, written by $\frac{1}{T}\int_{0}^{T}\frac{\partial}{\partial x}\mathbb{E}(\sigma_{s}^{2})ds$, we need to find an expression for $\frac{\partial}{\partial x}\mathbb{E}(\sigma_{s}^{2})$. To this end, we follow the technical approach described in this manuscript.\\
One can easily solves a Riccati's equation to obtain $\rho_{t}^{x}=\frac{x\mu e^{\mu t}}{q x e^{\mu t}+\mu - q x}$. Therefore, $\frac{\partial}{\partial x}\rho_{t}^{x}=\frac{e^{-\mu t}}{x^{2}}(\rho _{t}^{x})^{2}$, 
\begin{equation*}
\int _{0}^{t}\phi(t,s)\frac{\partial \Theta_s }{\partial x}ds= -\frac{qe ^{-\mu t}}{\alpha  x^{2}}(\rho _{t}^{x})^{2},
\end{equation*}
and 
\begin{equation*}
\frac{\partial}{\partial x}M_{t}^{x} = M_{t}^{x} \Big[ \int_0^t  \frac{\partial  \Theta _{u}^{x}}{\partial x}dW_{u}^{H} -\int _{0}^{t}\int_{0}^{t} \frac{\partial \Theta _{u}^{x}}{\partial x}\Theta _{\nu}^{x}\phi(\nu, u)du d\nu\Big].
\end{equation*}
Assume that $\mu > qx>0$. Theorem \ref{mt} shows that $\frac{\partial}{\partial x}\sigma _{t}=\langle D^H\sigma_t, h\rangle_\mathcal{H} + e_{t,X}^{-1}\mathcal{E}_t$ and $$R_Hh(t) = -\frac{q}{x^2}\alpha^{-1} \int_0^t e ^{-\mu s}(\rho _{s}^{x})^{2} ds.$$ 
In addition,  $w_t(u) = \frac{1}{K_0T\alpha} (t \wedge u) $ and therefore from (\ref{finit}),
$$( K_H^{-1}w_t )(u)= \frac{1}{K_0T\alpha} u^{\frac12-H}\Big(1_{[0,t]}(u)+ g_0\Big).$$
 Now, applying Lemma \ref{deltahkh} for $v=V_t^0$ and then Theorem \em{\ref{vaznn}} to reult 
\begin{align*}
\frac{\partial}{\partial x}\mathbb{E}^Q(\sigma _{t}^{2})&
=\mathbb{E}(\sigma_{t}^{2}\frac{\partial}{\partial x}{M}_{t}^{x} )+2\mathbb{E}(M_t^x  \sigma _{t}\frac{\partial}{\partial x}\sigma _{t})\\
&=\mathbb{E}(\sigma_{t}^{2}\frac{\partial}{\partial x}{M}_{t}^{x} )+\mathbb{E}\Big( M_t^x \langle D^H\sigma_t^2, h\rangle_\mathcal{H}\Big)+ \mathbb{E}\Big(2 \sigma _{t} M_t^x e_{t,X}^{-1}\mathcal{E}_t\Big)\\
&=\mathbb{E}(\sigma_{t}^{2}\frac{\partial}{\partial x}{M}_{t}^{x} ) +\mathbb{E}\Big( \sigma_{t}^2 M_t^x \delta^H(h)\Big) -\mathbb{E}\Big( \sigma_{t}^2 \langle D^HM_t^x, h\rangle_\mathcal{H}\Big)+ \mathbb{E}\Big( \sigma _{t}^2 \int_0^t M_t^x  K_H^{-1}(w_t)(s) dW_s\Big)\\
&=\mathbb{E}(\sigma_{t}^{2}\frac{\partial}{\partial x}{M}_{t}^{x} )-\frac{q}{x^2}\mathbb{E}\Big( \sigma _{t}^2 M_t^x \int_0^t K_H^{-1}\Big\{\int_0^. e^{-\mu u}(\rho_{u}^{x})^2 du \Big\}(s) dW_s\Big)\\
& - \mathbb{E}^Q( \sigma_{t}^2)\int_0^t\int_0^s \phi(s,u) \frac{\partial \Theta _{u}^{x}}{\partial x} du ds+  \frac{1+g_0}{K_0T\alpha}\mathbb{E}\Big( \sigma _{t}^2 \int_0^t M_t^x s^{\frac12-H} dW_s\Big).\\
\end{align*}
\begin{rem}
We note that if $\mathbb{E}(\sigma_{t})=\rho_{t}^{x}=0$ or $q=0$, then 
 \begin{equation*}
\frac{\partial}{\partial x}\mathbb{E}(\sigma _{t}^{2})= \mathbb{E}\Big(\sigma_{t}^{2} \Big[\frac{\partial}{\partial x}{M}_{t}^{x}+ \frac{1+g_0}{K_0T\alpha} \int_0^t  M_t^x s^{\frac12-H}  dW_s\Big] \Big).
\end{equation*}
\end{rem}
\end{example}
\begin{example}
Consider the following stochastic volatility model
\begin{equation*}
\left\{\begin{array}{lr}
S_t^{x_{1},x_{2}}=x_{1}+\int_{0}^{t}\mu S_{u}^{x_{1},x_{2}}du +\int_{0}^{t} g(\sigma_{u}^{x_{2}})S_{u}^{x_{1},x_{2}} dW_{u},\\
\sigma_{t}^{x_{2}}=x_{2}+\int_{0}^{t}(\mu -q \rho_{t}^{x_{2}})\sigma_{t}^{x_{2}}du+\int_{0}^{t} \alpha\sigma_{u}^{x_{2}}dW_{u}^{H}, \quad t \in [0,T],
\end{array}\right.
\end{equation*}
where  $\mu,q,\alpha\in\mathbb{R^+}$, $\mu > qx >0$ and $g: \mathbb{R} \rightarrow \mathbb{R}$ is a bounded measurable function with Lipschitz continuity property. For simplicity, denote by $S_t$ and $\sigma_t$, respectively, the solution of the above equation. We would like to investigate the sensitivity of the payoff functions of the form with respect to $x_2$. To do this, it follows from above example that 
$$\frac{\partial}{\partial x_2}\sigma _{t}=\langle D^H\sigma_t, h\rangle_\mathcal{H} + e_{t,X}^{-1}\mathcal{E}_t,$$ 
where $R_Hh(t) = -\frac{q}{x_2^2}\alpha^{-1} \int_0^t e ^{-\mu s}(\rho _{s}^{x_2})^{2} ds$. Also, from Theorem \ref{vaznn}, we know $\langle D^H\sigma_t, V_t^0\rangle_\mathcal{H} =  e_{t,X}^{-1}\mathcal{E}_t$. Thus, 
$$\frac{\partial}{\partial x_2}\sigma _{t}=\langle D^H\sigma_t, h+V_t^0\rangle_\mathcal{H}. $$
On the other hand, 
\begin{equation*}
\frac{\partial}{\partial x_{2}} S_{t}=\int_0^t\mu\frac{\partial}{\partial x_{2}} S_{s}ds +\int_0^t \{ g(\sigma_{s})
\frac{\partial}{\partial x_{2}}S_{s}+ S g'(\sigma_{s})\frac{\partial}{\partial x_{2}}\sigma_{s}\}dW_{s},\\
\end{equation*}
and
\begin{align*}
 D_{r}^{H} S_{t}& =\int _{0}^{t}\mu D_{r}^{H}S_{s}ds +\int_{0}^{t}\{g(\sigma_{s}) D_{r}^{H} S_{s} +S_{s}g'(\sigma _{s}) D_{r}^{H}\sigma_{s}\}dW_s,\\
\end{align*}
Then, since the solution is unique, $\frac{\partial}{\partial x_{2}} S_{t}=\langle D^{H} S_{t},  h+V_t^0\rangle_\mathcal{H}$.
Therefore, it follows for $X_t=(S_t, \sigma_t)$ that
\begin{align*}
\frac{\partial}{\partial x_{2}}\mathbb{E}\Big(\Phi_{s}^{\prime}(X_{t})\Big) &= \mathbb{E}\Big(\Phi_{s}^{\prime}(X_{t})\frac{\partial}{\partial x_{2}}S_{t}\Big)+\mathbb{E}\Big(\Phi'_{\sigma}(X_{t})\frac{\partial}{\partial x_{2}}\sigma_{t}\Big) \\
&=\mathbb{E}\Big(\Phi_{s}^{\prime}(X_{t}) \langle D^{H} S_{t},  h+V_t^0\rangle_\mathcal{H}\Big)
+\mathbb{E}\Big(\Phi_{\sigma}^{\prime}(X_{t})\langle D^{H} \sigma_{t},  h+V_t^0\rangle_\mathcal{H}\Big)\\
&=\mathbb{E}\Big( \langle D^{H}(\Phi(X_{t}) ) ,  h+V_t^0\rangle_\mathcal{H}\Big)\\
&=\mathbb{E}\Big(\Phi(X_{t})\delta(-\frac{q}{x^2} K_H^{-1}\Big\{\int_0^. e^{-\mu u}(\rho_{u}^{x})^2 du \Big\}+ K_H^{-1}w_t)\Big).
\end{align*}
\end{example}
\section{Conclusions}
In this work, we proved the Malliavin differentiability of the solution on mean-field stochastic differential equations whose dependence is the form of expection functional of the solution itself. We formulated the Bismut formula for the solution by changing the solution by another process whose Malliavin derivative and flow can be easily modified each other. We showed the application of the Bismut formula in the sensitivity analysis of variance swaps with distribution-dependent models.

\end{document}